%% file: main.tex
\documentclass[10pt]{article}

\usepackage[utf8]{inputenc}
\usepackage[T2A]{fontenc}

\usepackage[english]{babel}

\usepackage[margin=1.2cm]{geometry}
\usepackage{mathtools}

\usepackage{pgf,tikz,pgfplots}
\usepackage{mathrsfs}
\usetikzlibrary{arrows,calc,intersections,through,shapes,patterns}
\usetikzlibrary{decorations.text}
\usetikzlibrary{decorations.pathreplacing}
\usetikzlibrary{decorations.markings,arrows.meta}

\pgfplotsset{compat=1.15}
\usepackage{float}

\usepackage{amsmath,amsthm,amssymb}
\usepackage{hyperref}
\usepackage{authblk}

\newtheorem{theorem}{Theorem}[section]
\newtheorem{definition}[theorem]{Definition}

\newtheorem{lemma}[theorem]{Lemma}
\newtheorem{proposition}[theorem]{Proposition}
\newtheorem{corollary}[theorem]{Corollary}

\newtheorem{problem}[theorem]{Problem}

\newtheorem{question}[theorem]{Question}
\newtheorem{observation}[theorem]{Observation}

\theoremstyle{remark}

\DeclareMathOperator{\conv}{conv}

\newcommand\R{\mathbb{R}}

\newcommand{\F}{{\mathcal F}}
\newcommand{\St}{{\mathcal St}}

        \newcommand{\eps}{\varepsilon}
        \newcommand{\HH}{{\mathcal H}}
        \renewcommand{\H}{\HH^1}

        \newcommand{\forget}[1]{}

        \def\dist{\mathrm{dist}\,}
        
        \def\diam{\mathrm{diam}\,}

        \def\X{\mathrm{X}}
        \def\E{\mathrm{E}}
        \def\S{\mathrm{S}}
        \def\G{\mathrm{G}}
        
        \def\diam{\mathrm{diam}\,}

\title{Inverse maximal and average distance minimizer problems}

\author[1]{Mishanya Basok}
\author[2,4]{Danila Cherkashin}
\author[3,4]{Yana Teplitskaya}

\affil[1]{University of Helsinki, Finland}
\affil[2]{Institute of Mathematics and Informatics, Sofia}
\affil[3]{Mathematical Institute, Leiden University, the Netherlands}
\affil[4]{Chebyshev Laboratory, St.Petersburg}

\begin{document}

\maketitle

\begin{abstract}
Consider a compact $M \subset \mathbb{R}^d$ and $r > 0$. A maximal distance minimizer problem is to find a connected compact set $\Sigma$ of the minimal length, such that
\[
\max_{y \in M} \dist (y, \Sigma) \leq r.
\]
The inverse problem is to determine whether a given compact connected set $\Sigma$ is a minimizer for some compact $M$ and some positive $r$.

Let a Steiner tree $\St$ with $n$ terminals be unique for its terminal vertices. 
The first result of the paper is that $\St$ is a minimizer for a set $M$ of $n$ points and a small enough positive $r$.
It is known that in the planar case a general Steiner tree (on a finite number of terminals) is unique. 
It is worth noting that a Steiner tree on $n$ terminal vertices can be not a minimizer for any $n$ point set $M$ starting with $n = 4$;
the simplest such example is a Steiner tree for the vertices of a square.

It is known that a planar maximal distance minimizer is a finite union of simple curves. 
The second result is an example of a minimizer with an infinite number of corner points (points with two tangent rays which do not belong to the same line), which means that this minimizer can not be represented as a finite union of smooth curves.

Our third result is that every injective $C^{1,1}$-curve $\Sigma$ is a minimizer for a small enough $r>0$ and $M = \overline{B_r(\Sigma)}$.
The proof is based on analogues result by Tilli on average distance minimizers.
Finally, we generalize Tilli's result from the plane to $d$-dimensional Euclidean space.

\end{abstract}

\section{Introduction}

For a given compact sets $M, \Sigma \subset \mathbb{R}^d$ consider the functional
\[
	\F_{M}(\Sigma) := \sup _{y\in M}\dist (y, \Sigma),
\]
where $\dist$ stands for Euclidean distance.
The following problems appeared in~\cite{buttazzo2002optimal} and later has been studied in~\cite{miranda2006one,paolini2004qualitative}.
\begin{problem}[Maximal distance minimizing problem]
For a given compact set $M \subset \R^d$ and $r > 0$ to find a connected compact set $\Sigma$ of the minimal length (one-dimensional Hausdorff measure $\H$) such that
\[
\F_{M}(\Sigma) \leq r.
\]
\label{TheProblem}
\end{problem}


\begin{problem}[Average distance minimizing problem]
Let $\varphi: [0,+\infty)\to [0,+\infty)$ be any non-decreasing function. Given a bounded open set $\Omega\subset \mathbb R^d$ and a real number $l>0$ consider the problem of minimizing the functional
\[
\int_\Omega \varphi(\dist(x,\gamma))\,dx
\]
over all compact subsets $\gamma$ contained in the closure of $\Omega$ and having the length (1-dimensional Hausdorff measure) at most $l$.
\label{TheAverage}
\end{problem}

We call a solution of Problem~\ref{TheProblem} an $r$-\textit{minimizer} for $M$, and a solution of Problem~\ref{TheAverage} an \textit{average distance minimizer}.
Surveys on Problems~\ref{TheProblem} and~\ref{TheAverage} may be found in~\cite{cherkashin2022overview} and~\cite{lemenant2011presentationENG}, respectively.
Here we focus on the following inverse problems.
\begin{problem}
For a given connected compact set $\Sigma  \subset \mathbb{R}^d$ determine if it is a maximal distance minimizer for some compact $M$ and $r > 0$.
\label{InverseProblem}
\end{problem}

\begin{problem}
For a given connected compact set $\Sigma  \subset \mathbb{R}^d$ and a non-decreasing $\varphi: [0,+\infty)\to [0,+\infty)$ determine if $\Sigma$ is an average distance minimizer for some open bounded $\Omega$ and $l = \H(\Sigma)$.
\label{InverseAverageProblem}
\end{problem}

Let $B_\rho(O)$ stand for the open ball of radius $\rho$ centered at a point $O$, and let $B_\rho(T)$
be the open $\rho$-neighbourhood of a set $T$ i.e.
\[
B_\rho(T) := \bigcup_{x\in T} B_\rho(x).
\]
As usual, $\overline{T}$ stands for the closure of a set $T$. 
Note that the condition $\max_{y \in M} \dist (y, \Sigma) \leq r$ is equivalent to $M \subset \overline{B_r(\Sigma)}$.

We start with a simple observation.

\begin{observation}
\begin{itemize}
    \item[(i)] Let $\Sigma$ be an $r$-minimizer for some $M$. Then $\Sigma$ is an $r$-minimizer for $\overline{B_r(\Sigma)}$.
    \item[(ii)] Let $\Sigma$ be an $r$-minimizer for $\overline{B_r(\Sigma)}$. Then $\Sigma$ is an $r'$-minimizer for $\overline{B_{r'}(\Sigma)}$, for every $0 < r' < r$.
\end{itemize}
\label{Obs:MisBrSigma}
\end{observation}

\begin{proof}
Assume the contrary to the first item. Then $M \subset \overline{B_r(\Sigma)} \subset \overline{B_r(\Sigma')}$ for some connected compact $\Sigma'$ such that
$\HH(\Sigma') < \HH(\Sigma)$, which contradicts the fact that $\Sigma$ is an $r$-minimizer for $M$. 

Assume the contrary to the second item. Then $\overline{B_{r'}(\Sigma)} \subset \overline{B_{r'}(\Sigma')}$ for some connected compact $\Sigma'$ such that
$\HH(\Sigma') < \HH(\Sigma)$. Since $\overline{B_{a+b}(X)} = \overline{B_a(B_b(X))}$ for an arbitrary set $X \subset \mathbb{R}^d$ and $a,b > 0$, 
\[
\overline{B_{r}(\Sigma)} = \overline{B_{r-r'}(B_{r'}(\Sigma))} \subset \overline{B_{r-r'}(B_{r'}(\Sigma'))} = \overline{B_{r}(\Sigma')}.
\]
Thus $\Sigma$ is not an $r$-minimizer for $\overline{B_{r}(\Sigma)}$; a contradiction.
\end{proof}

This motivates the following open question.
\begin{question}
Let $\Sigma$ be an $r$-minimizer for some $M$. Is $\Sigma$ the unique $r$-minimizer for $\overline{B_r(\Sigma)}$?
\end{question}
A weaker form of this question is if we replace $r$ with some positive $r_0 < r$ in the hypothesis.

Also we need some basic facts on $r$-minimizers (see~\cite{miranda2006one} for planar $M$ and~\cite{PaoSte04max} for an arbitrary dimension). 
Let $\Sigma$ be a minimizer for a compact set $M \subset \mathbb{R}^d$ and $r > 0$.
A point $x \in \Sigma$ is called \textit{energetic} if for every $\varepsilon > 0$ the inequality
\[
\F_{M}(\Sigma \setminus B_{\varepsilon}(x)) > \F_{M}(\Sigma)
\]
holds. The set of all energetic points of $\Sigma$ is denoted by $\G_\Sigma$.
Every $r$-minimizer $\Sigma$ can be split into three disjoint subsets:
\[
        \Sigma=\E_{\Sigma}\sqcup\X_{\Sigma}\sqcup\S_{\Sigma},
\]
where $\X_\Sigma \subset \G_\Sigma$ is the set of \textit{isolated energetic} points (i.e. every $x \in \X_\Sigma$ is energetic and there is a $\rho > 0$ such that $B_\rho(x) \cap \G_\Sigma = \{x\}$), $\E_\Sigma := \G_\Sigma \setminus \X_\Sigma$ is the set of \textit{non-isolated energetic}
points and $\S_\Sigma := \Sigma \setminus \G_\Sigma$ is the set of non-energetic points also called the \textit{Steiner part} of $\Sigma$.

The following statements will be further referred as \textit{basic properties}       
\begin{itemize}
    \item[(a)] minimizers contain no cycles (homeomorphic images of circumference). 
    \label{a}
    \item[(b)] For every energetic $x \in \G_\Sigma$ there is a point $y \in M$, such that $|x-y|=r$ and $B_{r}(y)\cap \Sigma=\emptyset$. 
    Further we call $y$ \textit{corresponding} to $x$ and denote by $y(x)$. Note that a corresponding point may be not unique.
    \label{b}
    \item[(c)] For every non-energetic $x \in \S_\Sigma$ there is an $\varepsilon > 0$, such that $\Sigma \cap B_{\varepsilon}(x)$ is either a segment or a \textit{regular tripod}, i.e. the union of three line segments with an endpoint in $x$ and relative angles of $2\pi/3$.
    \label{Sbasis}
\end{itemize}

If $x \in \S_\Sigma$ is the center of a regular tripod, we call it \textit{branching point} of $\Sigma$.
\begin{definition}\label{def:tgray}
We say that the ray $ (ax] $ is a \textit{tangent ray} of a set $ \Gamma \subset \mathbb{R}^d $
at a point $ x\in \Gamma $ if there exists a
sequence of points $ x_k \in \Gamma \setminus \{x\}$ such that $ x_k \rightarrow x $ and $ \angle x_kxa \rightarrow 0 $.
\end{definition}
\begin{definition}
We will say that the ray $ (ax] $ is a \textit{one-sided tangent} of a set $ \Gamma  \subset \mathbb{R}^n $ at a point $ x \in \Gamma $ if there exists a connected component $\Gamma_1$ of $\Gamma \setminus \{x\}$ such that $x \in \overline{\Gamma_1}$ and that any sequence of points $x_k \in \Gamma_1$ with the property $x_k \rightarrow x$ satisfies $\angle x_kxa \rightarrow 0$.
In this case we will also say that $(ax]$ is tangent to the connected component $\Gamma_1$.
\end{definition}

\begin{theorem}[Gordeev--Teplitskaya~\cite{gordeev2022regularity}]
\label{theoGT}
Let $\Sigma$ be a solution of Problem~\ref{TheProblem} for a compact set $M \subset \mathbb{R}^d$ and $r > 0$. Assume that $\Sigma$ is not a point. Then $\Sigma$ has the following properties:
\begin{enumerate}
    \item For each point $x\in \Sigma$ the complement $\Sigma\smallsetminus\{x\}$ has at most 3 connected components and the closure of each connected component has a unique tangent ray (which is a one-sided tangent therefore) at $x$. If the number of connected components is 3, then the angle between each pair of tangent rays is $2\pi /3$. If the number of components is 2, then the angle between the two tangent rays is at least $2\pi /3$. Particularly angles between one-sided tangents can not be equal to $0$.
    \item If $d = 2$, then $\Sigma$ can be written as a union of simple curves $\gamma_1,\ldots,\gamma_k$, where each $\gamma_i$ has one-sided tangent continuous from the corresponding side. More precisely, for each $i = 1,\dots, k$ and continuous parametrization $\gamma_i: [0,1]\to \mathbb R^2$ there exists a function $v_i: (0,1]\to S^1$ which is continuous from the left and such that for each $t\in (0,1]$ the ray starting at $\gamma_i(t)$ in the direction $v_i(t)$ is a left-sided tangent to $\gamma_i$, where ``left-sided'' is understood with respect to the orientation given by the parametrization.
\end{enumerate}

\end{theorem}

\paragraph{Structure of the paper.} The rest part of the introduction collects the notation and enlists definitions and some basic results in the Euclidean Steiner tree problem. Section 2 contains the results, Sections 3--5 contains proofs of the first, the second and the third results, mentioned in the abstract, respectively.

\subsection{Notation}

For given points $b$, $c$ we use the notation $[bc]$, $[bc)$ and $(bc)$ for the corresponding closed line
segment, ray and line respectively. We also use $[bc[$ for the semiopen interval $[bc] \setminus \{c\}$.

Recall that $B_\rho(T)$ stands for the open $\rho$-neighbourhood of a set $T$. By $\overline{T}$, $\partial T$ and $\conv T$ we denote the closure, the boundary and the convex hull of a set $T$, respectively. 

Further $\omega_k$ denotes the volume of the unit ball in $\mathbb{R}^k$.

\subsection{Steiner trees}

\label{subsect:introSt}

We need the following form of the Steiner (tree) problem in a Euclidean space:
\begin{problem}\label{Problem1}
For a given finite set $P = \{x_1,\dots ,x_n\} \subset \mathbb{R}^d$ to find a connected set $\St(P)$ with the minimal length (one-dimensional Hausdorff measure) containing $P$.
\end{problem}

A solution of Problem~\ref{Problem1} is called \textit{Steiner tree}.
It is known that such an $\St = \St(P)$ always exists (but is not necessarily unique) and that it is a union of a finite set of segments. Moreover, $\St$ can be represented as a graph, embedded into the Euclidean space, such that its set of vertices contains $P$ and all its edges are straight line segments. This graph is connected and does not contain cycles, i.e. is a tree, which explains the naming of $\St$. It is known that the maximal degree of the vertices of $\St$ is at most $3$. Moreover, only vertices $x_i$ can have degree $1$ or $2$, all the other vertices have degree $3$ and are called \textit{Steiner points} while the vertices $x_i$ are called \textit{terminals}. 
Vertices of the degree $3$ are called \textit{branching points}.
The angle between any two adjacent edges of $\St$ is at least $2\pi/3$.
That means that for a branching point the angle between any two segments incident to it is exactly $2\pi/3$, and these three segments belong to the same 2-dimensional plane.

The number of Steiner points in $\St$ does not exceed $n-2$. A Steiner tree with exactly $2n-2$ vertices is called \textit{full}.  
Every terminal point of a full Steiner tree has degree one.

For a given finite set $P$ consider a connected acyclic set $S$ containing $P$.
Then $S$ is called a \textit{locally minimal tree} if $\overline{S \cap B_\varepsilon (x)}$ is a Steiner tree for 
$(\{x\} \cap P) \cup (S \cap \partial B_\varepsilon (x))$ for every point $x \in S$ and small enough $\varepsilon>0$. Clearly every Steiner tree is locally minimal and not vice versa.
Locally minimal trees have all the mentioned properties of Steiner trees except the minimal length condition. So locally minimal trees inherit the definitions of terminals, Steiner points, branching points and fullness.
A proof of the listed properties of Steiner and locally minimal trees together with an additional information on them can be found in
book~\cite{hwang1992steiner} and in article~\cite{gilbert1968steiner}.

The Steiner problem may have several solutions starting with $n = 4$ (see Fig.~\ref{fig1}). It is known~\cite{basok2018uniqueness} that for $n \geq 4$ the set of $n$-point configurations for which the solution of the planar Steiner problem is not unique has the Hausdorff dimension at most $2n-1$ (as a subset of $\mathbb{R}^{2n}$).

\begin{figure}[h]
    \centering
    \input{pictures/1nonunique.tex}
    \caption{An example of non-unique solution. Labelled points form a square.}
    \label{fig1}
\end{figure}
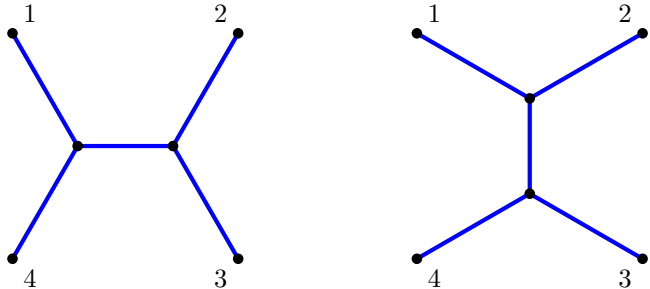

A \textit{topology} $T$ of a labelled Steiner tree (or a labelled locally minimal tree) $\St$ is the corresponding abstract graph with labelled terminals and unlabelled Steiner points. Further, let us call a topology $T$ \textit{realizable} for a set $P$ if there exists such a locally minimal tree $S(P)$ with topology $T$; we will denote this tree by $S_T(P)$.

\begin{proposition}[Melzak,~\cite{melzak1961problem}]
If a topology $T$ is realizable for $P$ then the realization $S_T(P)$ is unique.
\label{melzakuniq}
\end{proposition}
Proposition~\ref{melzakuniq} shows that $S_T(P)$ is uniquely defined.

\section{Results}

\subsection{\texorpdfstring{$C^{1,1}$}{C11} curves maximize the volume of \texorpdfstring{$R$}{R}-neighborhood}
\label{subsec:C11-curves-theorem}

The proof of the following folklore inequality can be found, for instance in~\cite{mosconi2005gamma}.

\begin{lemma}
\label{lemma:area_inequality}
Let $\gamma$ be a compact connected subset of $\mathbb{R}^d$ with $\H(\gamma) < \infty$.
Then
\[
\HH^d(\{x \in \mathbb{R}^d: \dist(x,\gamma) \leq t\} ) \leq \H (\gamma) \omega_{d-1}t^{d-1} + \omega_d t^d,
\]
where $\omega_k$ denotes the volume of the unit ball in $\mathbb{R}^k$.
\end{lemma}

One can ask when we have equality in Lemma~\ref{lemma:area_inequality}. It is easy to show that in this case $\gamma$ must be a simple curve. Further examination may show that $\gamma$ should be $C^1$; to see this heuristically, one can observe that the inequality in Lemma~\ref{lemma:area_inequality} becomes strict for any broken line with at least two segments. Furthermore, an example of an arc of a circle of a small radius suggests that $\gamma$ must have the curvature radius at least $t$. 
The first aim of the current work is to make these heuristics into a theorem characterizing all curves in $\mathbb R^d$ for which the inequality in Lemma~\ref{lemma:area_inequality} becomes an equality.

Let $\gamma\subset \mathbb R^d$ be a rectifiable curve, and by an abuse of the notation let also $\gamma:[0,\H(\gamma)]\to \mathbb R^d$ denote its arc length parametrization. We say that $\gamma$ has the curvature radius at least $R$ if $\gamma\in C^1$ and for any $s,t\in (0,\H(\gamma))$ we have $|\gamma'(s) - \gamma'(t)|\leq R^{-1}|s-t|$.

\begin{theorem}
    \label{theorem:c11}
    Let $\gamma\subset \R^d$ be a rectifiable curve and $R>0$ be given. Then the following conditions are equivalent:
    \begin{itemize}
        \item [(i)] We have $\HH^d(B_R(\gamma) ) = \H (\gamma) \omega_{d-1}R^{d-1} + \omega_d R^d$, where $\omega_k$ is the volume of the unit ball in $\mathbb R^k$.
        \item [(ii)] For any $p\in B_R(\gamma)$ there exists a unique $t\in [0,\H(\gamma)]$ such that $|p-\gamma(t)| = \dist(p,\gamma)$.
        \item [(iii)] The curve $\gamma$ has the curvature radius at least $R$ and for any $p\in B_R(\gamma)$ there exists a unique $t\in [0,\H(\gamma)]$ such that $|p-\gamma(t)| = \dist(p,\gamma)$.
    \end{itemize}
\end{theorem}

We prove Theorem~\ref{theorem:c11} in Section~\ref{sec:appendix}.

The next statement has been proven by Tilli in~\cite{tilli2010some} in the planar case; using Theorem~\ref{theorem:c11} and generalizing some of Tilli's arguments it turns easy to obtain the following:
\begin{corollary}
    \label{cor:length+curv}
    Assume that $\gamma\subset \mathbb R^d$ is a rectifiable curve of length at most $\pi R$ and curvature radius at least $R$ for some $R>0$. Then $\gamma$ satisfies the condition from the item~(ii) of Theorem~\ref{theorem:c11} with the given $R$, in particular, $\HH^d(B_R(\gamma) ) = \H (\gamma) \omega_{d-1}R^{d-1} + \omega_d R^d$.
\end{corollary}

Recall that a curve $\gamma = \gamma(t)$ parameterized by its arc length is called $C^{1,1}$, if $\gamma'(t)$ is $C$-Lipshitz for some constant $C$.

\begin{corollary}
    \label{cor:C11_r_small_enough}
    For each simple $C^{1,1}$ curve $\gamma$ in $\mathbb R^d$ there exists $\eps>0$ such that $\gamma$ satisfies the condition from the item~(ii) of Theorem~\ref{theorem:c11} with $R=\eps$, in particular, $\HH^d(B_\eps(\gamma) ) = \H (\gamma) \omega_{d-1}\eps^{d-1} + \omega_d \eps^d$.
\end{corollary}

We prove Corollaries~\ref{cor:length+curv} and~\ref{cor:C11_r_small_enough} in Section~\ref{sec:appendix}.
Note that Corollary~\ref{cor:length+curv} permits us to easily generalize Tilli's result~\cite[Theorem~1.1]{tilli2010some}. 
\begin{theorem}
    \label{theorem:tilli_in_Rd}
    Let $\gamma\subset \mathbb R^d$ be a rectifiable curve of length $l$ and satisfying the condition from item~(ii) of Theorem~\ref{theorem:c11} with some $R$; for example, $\gamma$ is as in Corollary~\ref{cor:length+curv} or~\ref{cor:C11_r_small_enough}. Then the curve $\gamma$ is a solution of Problem~\ref{TheAverage} for the set $\Omega = B_R(\gamma)$ and any $\varphi$.
\end{theorem}
\begin{proof}
    Having Theorem~\ref{theorem:c11} in our pocket, we can simply repeat the arguments from~\cite{tilli2010some}. Without loss of generality we can assume that $\varphi(0) = 0$ and $\varphi$ is continuous from the left. For the last property note that for any non-zero $t$ the level set $\{x\in \Omega\ \mid\ \dist(x,\beta) = t\}$ has zero Lebesgue measure, and $\varphi$ has only countably many discontinuity points; thus, making $\varphi$ continuous from the left amounts in changing the function $\varphi(\dist(x,\beta))$ on a set of measure zero and does not affect the integral in the Problem~\ref{TheAverage}. Under the assumptions above we can write
    \[
        \varphi(t) = \mu([0,t))
    \]
    where $\mu$ is a Borel measure on $\mathbb R$. Given a compact set $\beta$ put
    \[  
        F_\beta(t) = \mathcal{H}^d(\{x\in \Omega\ \mid\ \dist(x,\beta)\leq t\}).
    \]
    Assume that the length of $\beta$ is at most $l$. We have by Fubini's theorem and Lemma~\ref{lemma:area_inequality}
    \begin{multline*}
        \int_\Omega \varphi(\dist(x,\beta))\,dx = \int_\Omega \left(\int_0^{+\infty} \chi_{[0,\dist(x,\beta))}(t)\,d\mu(t)\right)\,dx = \int_0^{+\infty} (\mathcal{H}^d(\Omega) - F_\beta(t))\,d\mu(t) \geq \\
        \geq \int_0^{+\infty} (\mathcal{H}^d(\Omega) - \min(\mathcal{H}^d(\Omega), l \omega_{d-1}t^{d-1} + \omega_d t^d)\,d\mu(t).
    \end{multline*}
    By Theorem~\ref{theorem:c11} the last inequality becomes an equality when $\beta = \gamma$, which concludes the proof.
\end{proof}

\subsection{Bounds on the length of an \texorpdfstring{$r$}{r}-minimizer}

Lemma~\ref{lemma:area_inequality} and Theorem~\ref{theorem:c11} can be also used to study the properties of maximal distance minimizers, as we demonstrate in the next subsections.

Let us provide a general lower bounds on the length of an $r$-minimizer for a given $M$. Clearly, if a set $\Sigma$ attains the equality in a lower bound for some $M$ and $r$, then it is an $r$-minimizer for $M$.

\begin{corollary}
\label{TilliType}
A maximal distance $r$-minimizer for a set $M$ has the length at least
\[
\max \left (0, \frac{\HH^d(M)  - \omega_d r^d}{\omega_{d-1}r^{d-1}} \right).
\]
\end{corollary}

Recall that a curve $\gamma$ is $C^{1,1}$-curve if $\gamma$ has the curvature radius at least $t$ for some $t>0$. Corollaries~\ref{TilliType} and~\ref{cor:C11_r_small_enough} imply the following:

\begin{corollary}
\label{cor:Tillis}
Let $\gamma$ be a simple $C^{1,1}$-curve. Then $\gamma$ is a solution of Problem 1 for a small enough $r$.
\end{corollary}

Let us recall the Go{\l}{\k{a}}b's theorem (see~\cite[Theorem~10.19]{MorelSolimini}):
\begin{theorem}[Go{\l}{\k{a}}b]
Let $\Sigma_1,\Sigma_2,\dots$ be a sequence of connected compacts in $\mathbb R^d$ converging to a compact $\Sigma$ with respect to the Hausdorff distance. Then
\[
    \H(\Sigma) \leq \liminf_{k\to \infty} \H(\Sigma_k).
\]
    \label{theorem:Golabs}
\end{theorem}

The following fact is proven in~\cite{fekete1998traveling} for the case when the set $\Sigma$ is a finite union of broken lines (see also~\cite[Corollary~2.3]{fekete1998traveling}):
\begin{lemma}
    \label{lemma:fekete}
    Assume that $\Sigma\subset \mathbb R^2$ is a compact connected subset of finite length. Then for any $r>0$
    \[
    \H(\partial B_r(\Sigma))\leq 2\H(\Sigma)+2\pi r.
    \]
\end{lemma}
\begin{proof}
    Suppose first that $\Sigma$ is a finite union of broken lines. Then by~\cite[Theorem~2.2]{fekete1998traveling} there exists a closed curve $\gamma$ such that $\partial B_r(\Sigma)\subset \gamma$ and
    \[
        \H(\gamma) \leq 2\H(\Sigma) + 2\pi r. 
    \]
    Now assume that $\Sigma$ is arbitrary. We approximate $\Sigma$ with unions of broken lines as follows. For a given $k>0$ let $\mathcal{E}_k\subset \Sigma$ be a finite $1/k$-net and $\Sigma_k$ be an arbitrary solution of the Steiner problem for $\mathcal{E}_k$. By Subsection~\ref{subsect:introSt} $\Sigma_k$ exists and is a finite tree embedded by straight lines. By the definition we have
    \[
        \H(\Sigma_k)\leq \H(\Sigma).
    \] 
    On the other hand, for any subsequential limit (with respect to the Hausdorff distance) $\Sigma'$ of the sequence $\Sigma_k$ we have $\Sigma\subset \Sigma'$ and so
    \[
        \H(\Sigma)\leq \H(\Sigma')\leq \liminf_{k\to \infty} \H(\Sigma_k)
    \]
    by Go{\l}{\k{a}}b's theorem. It follows that $\Sigma_k$ converges to $\Sigma$ and $\H(\Sigma_k)$ converges to $\H(\Sigma)$.

    Let now $\gamma_k$ be a closed curve containing $\partial B_r(\Sigma_k)$ and such that
    \[
        \H(\gamma_k)\leq 2\H(\Sigma_k) + 2\pi r.
    \]
    It is straightforward to see that any subsequential limit of $\gamma_k$ contains $\partial B_r(\Sigma)$. It follows from the Go{\l}{\k{a}}b's theorem and the definition of $\gamma_k$ that
    \[
        \H(\partial B_r(\Sigma))\leq \liminf_{k\to \infty}\H(\gamma_k) \leq 2\H(\Sigma) + 2\pi r
    \]
    as required.
\end{proof}

\begin{corollary}
Let $M\subset \mathbb{R}^2$ be a convex compact set and $\Sigma$ be an $r$-minimizer for $M$. Then 
\[
\H(\Sigma) \geq \frac{\H(\partial M) - 2\pi r}{2}.
\]
\label{Cor:perimeter}
\end{corollary}

\begin{proof}
Since $M \subset \overline{B_r(\Sigma)} \subset \conv \overline{B_r(\Sigma)}$ and both $M$ and $\conv \overline{B_r(\Sigma)}$ are convex one may use a well-known application of the Cauchy’s formula (see~\cite{pach2011combinatorial}) to get
\[
\H(\partial M)\leq \H (\partial \conv B_r(\Sigma)).
\]
Clearly,
\[
\H (\partial \conv B_r(\Sigma)) \leq \H (\partial B_r(\Sigma))
\]
and by Lemma~\ref{lemma:fekete} 
\[
\H (\partial B_r(\Sigma))\leq 2\H(\Sigma)+2\pi r.
\]
Summing up
\[
\H(\partial M)\leq \H (\partial \conv B_r(\Sigma)) \leq \H (\partial B_r(\Sigma))\leq 2\H(\Sigma)+2\pi r.
\]
\end{proof}

It is worth noting that the equality in Corollary~\ref{Cor:perimeter} reaches if and only if $\Sigma$ is a segment or a point and $M = \overline{B_r(\Sigma)}$.

\subsection{Maximal distance minimizers for \texorpdfstring{$n$}{n} points}

Here we analyze whether a set $\Sigma$ is a minimizer for $M \subset \mathbb{R}^d$ consisting of $n$ points $a_1, \dots a_n$. The following observation uses definitions and straightforward local arguments (in fact we may replace $B_r(a_i)$ with a strictly convex $C^1$-smooth set). 

\begin{observation}
\begin{itemize}
    Let $\Sigma$ be an $r$-minimizer for an $n$-point set $M = \{a_1,\dots, a_n\}$. 
    \item[(i)] Then $\Sigma$ is a Steiner tree for $\{b_1,\dots,b_n\}$, where $b_i$ is an arbitrary point from $\overline{B_r(a_i)} \cap \Sigma$.
    \item[(ii)] Let $ b_i \in B_r(a_i) \cap \Sigma$ be a point that does not belong to $\overline{B_r(a_j)}$, $i \neq j$.
    Then $B_i$ is an inner point of a segment of $\Sigma$ or a branching point of $\Sigma$.
    \item[(iii)] Let $b_i \in \partial B_r(a_i) \cap \Sigma$ be a point that does not belong to $\overline{B_r(a_j)}$, $i \neq j$. 
    \begin{itemize}
        \item[a)] If $b_i$ has degree 1, then it is an end of a segment of $\Sigma$, which is contained in $[a_ib_i)$, i.e. is orthogonal to $\partial B_r(a_i)$.
        \item[b)] If $b_i$ has degree 2, it may be an inner point of a segment of $\Sigma$ or $[a_ib_i)$ contains 
        the bisector of the angle $B_i$ in $\Sigma$.
        \item[c)] If $b_i$ has degree 3 then it is a branching point of $\Sigma$.
    \end{itemize}
    
    \item[(iv)] Let $r' < r$ be a positive real. Then $\Sigma$ is an $r'$-minimizer for some $n$-point set $M'$. 
    
\end{itemize}
\label{Obs:local}
\end{observation}

Let $\rho$ stands for the Chebyshev distance between $n$-tuples $v = (v_1,\dots,v_n)$ and $u = (u_1,\dots, u_n)$, i.e.
\[
\rho(u,v) = \max_{1 \leq i \leq n} \dist (v_i,u_i),
\]
where $v_i$, $u_i$ are points in the space.

\begin{theorem}
Let $\St$ be a Steiner tree for terminals $A = (a_1,\dots, a_n)$, $a_i \in \mathbb{R}^d$ such that every Steiner tree for an $n$-tuple in the closed $2r$-neighbourhood of $A$ (with respect to Chebyshev distance) has the same topology as $\St$ for some positive $r$. Then $\St$ is an $r$-minimizer for an $n$-tuple $M$.
\end{theorem}

\begin{proof}
Note that the condition on $\St$ implies that $B_r(a_i) \cap B_r(a_j)$ is empty for $i \neq j$, otherwise one may shift $a_i$ and $a_j$ into the same point making a tree with a topology different from the topology of $\St$. 
Also by Proposition~\ref{melzakuniq} $\St$ is the unique Steiner tree for its terminals.

Construct $M = \{b_1,\dots,b_n\}$ as follows. 
If $a_i$ has degree 1, i.e. $B_\varepsilon (a_i) \cap \St$ is a segment $[a_it[$ for a small enough $\varepsilon > 0$, then $b_i$ is an intersection of ray $[ta_i)$ with $\partial B_r(a_i)$. 
If $a_i$ has degree 2, then $b_i$ is the point in the outer bisector of $\angle a_i$ at the distance $r$ from $a_i$.
Terminal $a_i$ can not have degree 3, otherwise any perturbation of $a_i$ preserving other terminals changes the topology, which contradicts the stability condition of the theorem.

Let $\Sigma$ be an $r$-minimizer for $\{b_1,\dots,b_n\}$. By Observation~\ref{Obs:local}(i) $\Sigma$ is a Steiner tree for some points $c_1,\dots c_n$, where $c_i \in \overline{B_r(b_i)}$.
By the triangle inequality $\dist(a_i,c_i) \leq \dist(a_i,b_i) + \dist(b_i,c_i) = 2r$, so one has $\rho(\{a_i\},\{c_i\}) \leq 2r$.
By the stability assumption $\Sigma$ has the same topology as $\St$. 

We need the following argument, which uses a convexity of length in the vein of~\cite{gilbert1968steiner} to show that a realization of a topology is unique.

\begin{lemma}
Let $\{\mathcal{A}_i\}_{i=1}^n$ be a family of $n$ disjoint strictly convex $C^1$-smooth closed sets in $\mathbb{R}^d$, $T$ be a topology with $n$ terminals. Then all trees with the topology $T$ such that for every $i$ the terminal with label $i$ belongs to $\mathcal{A}_i$ with locally minimal length coincide as a (non-labelled) subsets of $\mathbb{R}^d$. 
\label{keylemma}
\end{lemma}

Lemma~\ref{keylemma} is proved in Section 3.

Application of Lemma~\ref{keylemma} with $\mathcal{A}_i = \overline{B_r(b_i)}$ shows that there is unique a local minimizer. By the construction it is $\St$. 

\end{proof}

\begin{proposition}
Let $\St$ be a full Steiner tree for terminals $a_1,\dots, a_n \in \mathbb{R}^d$, and let $q$ be the smallest distance between a terminal $a_i$ and the branching point adjacent to $a_i$ among $i=1, \dots, n$.
Then $r < q$ implies that $\St \setminus \cup_{i = 1}^n B_r(a_i)$ is an $r$-minimizer for $M = \{a_1,\dots, a_n\}$. 
\end{proposition}

\begin{proof}
Since $\St$ is full every $a_i$ has degree 1.
Note that the intersection of $\partial B_r(a_i)$ and $\St$ consists of a unique point because $r < q$ and $\St$ is a Steiner tree; call this point $b_i$.
Note that $\St_1 := \St \setminus B_r(a_i)$ is a connected tree containing points $b_i$ and
\[
\H (\St_1) = \H (\St) - nr.
\]
If there is a shorter tree $\St_2$ connecting some points $c_i \in \overline{B_r(a_i)}$, then set $\St_3 := \St_2 \cup [a_1c_1] \cup \dots [a_nc_n]$
connects points $a_i$ and 
\[
\H (\St_3) \leq \H (\St_2) + nr < \H (\St_1) + nr = \H (\St)
\]
which contradicts the fact that $\St$ is a Steiner tree for $\{a_i\}$.
\end{proof}

We need the following theorem to extend the example of square, mentioned in the abstract.

\begin{theorem}[Oblakov~\cite{oblakov2009non}]
There are no two distinct topologies $T_1$ and $T_2$ and a planar configuration $P$ such that locally minimal trees $S_{T_1}(P)$ and $S_{T_2}(P)$ are codirected at terminals.
\label{oblakov}
\end{theorem}

\begin{corollary}
Suppose that $\St$ is a full Steiner tree for terminals $a_1,\dots, a_n \in \mathbb{R}^2$, which is not unique. 
Then $\St$ can not be a minimizer for $M$ being an $n$-tuple of points. 
\label{cor:notaminimizer}
\end{corollary}

It is interesting whether the fullness condition in Corollary~\ref{cor:notaminimizer} is neсessary.

\begin{proof}
Denote by $\St' \neq \St$ an arbitrary Steiner tree for $a_1,\dots, a_n$.
Assume the contrary, i.e. $\St$ is an $r$-minimizer for some $n$-tuple $M = \{b_1, \dots, b_n\}$. Then for every $i$ the segment of $\St$ ending at $a_i$ belongs to the same line with $b_i$. Note that if $i$ has degree 1 in $\St'$ then $\St$ is codirected with $\St'$ at $a_i$.
We claim that if $i$ has degree 2 in $\St'$ then the angles between $\St$ and $\St'$ at $a_i$ are $\pi/3$. 
The claim implies that 
\[
\St \cup \bigcup_{i=1}^n ([a_ib_i] \cap \overline {B_\varepsilon(a_i)}) \quad \quad \mbox{and} \quad \quad  \St' \cup \bigcup_{i=1}^n ([a_ib_i] \cap \overline {B_\varepsilon(a_i)})
\]
are codirected locally minimal trees for a small enough $\varepsilon > 0$. By Theorem~\ref{oblakov} they have the same topology and then by Proposition~\ref{melzakuniq} they coincide which implies $\St = \St'$.

Now let us prove the claim. Let $S_1,\dots, S_k$ be full components of $\St'$; we call $S_i$ and $S_j$ \textit{adjacent} if they share a terminal.
By the construction all segments of full Steiner tree $\St$ are parallel to one of lines $l_1$, $l_2$, $l_3$ having all pairwise angles equal to $\pi/3$.
Since $\St'$ has a terminal of degree 1 (without loss of generality it is contained in $S_1$), all segments of $S_1$ are also parallel to one of lines $l_1$, $l_2$, $l_3$ and then the claim holds for every terminal of $S_1$. 
By Observation~\ref{Obs:local}~(iii)b if all segments of $S_i$ are parallel to one of lines $l_1$, $l_2$, $l_3$, then the same holds to every adjacent $S_j$. We are done, because $\St'$ is connected.
\end{proof}

The situation in the case of several non-full solutions remains open.

\subsection{An example of a maximal distance minimizer with an infinite number of corner points}

Recall that by Theorem~\ref{theoGT} in the planar case a maximal distance minimizer is a finite union of simple curves. It however remained unclear if these curves can be chosen to be $C^1$-smooth.
Note that the existence of a tangent line at each point of each curve would imply the $C^1$ continuity of the latter, because Theorem~\ref{theoGT} guarantees the continuity of one-sided tangents.

In the theorem below we show that the answer to the question above is negative: we will provide an explicit example of the maximal distance minimizer having a limit point of the corner points (points with two tangent rays which do not belong to the same line).

Fix positive reals $r$, $R$ and let $N$ be a large enough integer. Consider a sequence of points $\{a_i\}_{i=1}^\infty$ chosen from the circumference $\partial B_R(o)$ such that
$N \cdot |a_2a_1|=r$, 
\[
|a_{i+1}a_{i+2}| = \frac{1}{2}|a_ia_{i+1}|
\]
and $\angle a_ia_{i+1}a_{i+2} > \frac{\pi}{2}$ for every $i \in \mathbb{N}$ (see Fig.~\ref{Fig:cornerexample}). Let $a_\infty$ be the limit point of $\{a_i\}$. Finally, let $a_{\infty+1}$ be the point in the tangent line to $B_r(o)$ at $a_\infty$, such that 
\[
|a_\infty a_{\infty+1}| = r/N.
\]

We claim that polyline 
\[
\Sigma = \bigcup_{i=1}^{\infty} a_ia_{i+1}
\]
is a unique maximal distance minimizer for the following $M$.

Let $v_1 \in (a_1a_2]$ be such point that $|v_1a_1| = r$.
For $i\in \mathbb{N}\cup\{\infty\} \setminus \{1\}$ define $v_i$ as the point satisfying $|v_ia_i|=r$ and $\angle a_{i-1}a_iv_{i}=\angle a_{i+1}a_iv_{i}>\pi/2$. Define $v_\infty$ as the limit point of $\{v_i\}$.
Finally, let $v_{\infty + 1}$ be such point that $v_{\infty + 1}a_\infty \perp v_\infty a_\infty$ and $|v_{\infty+1} a_\infty| = r$. 
Clearly $M := \{v_i\}_{i=1}^{\infty+1}$ is a compact set.

\begin{figure}[h]
        \centering
        \input{pictures/2corner.tex}
        \caption{The example of a minimizer with infinite number of corner points}
        \label{Fig:cornerexample}
\end{figure}
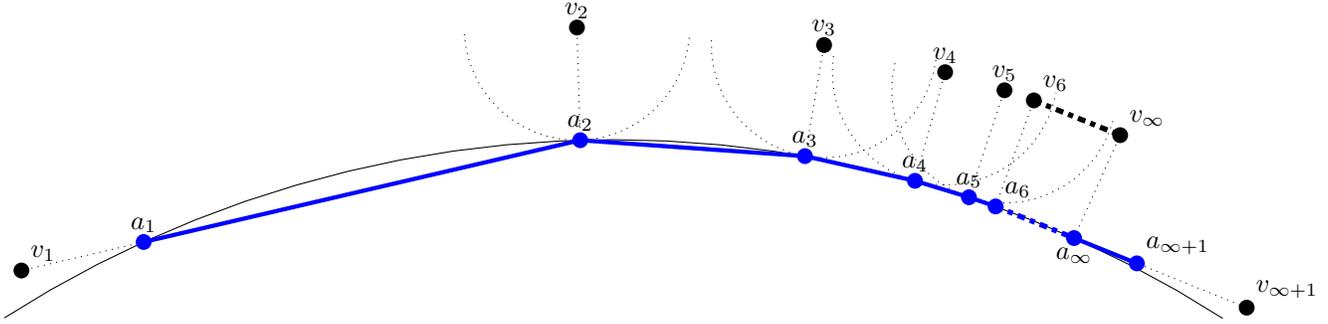

\begin{theorem}
Let $\Sigma$ and $M$ be as defined above. Then $\Sigma$ is the unique maximal distance minimizer for $M$.
\label{theocornerpoints}
\end{theorem}

\section{Proof of Lemma~\ref{keylemma}}

\begin{proof}
Assume the contrary. Let $n$ be the smallest possible; it implies that there is no terminal of degree 3. 
Let $k$ denote the number of Steiner points in $T$. Given $n+k$ points $v_1,\dots, v_{n+k}\in \mathbb R^d$, denote by $\bar v := (v_1,v_2, \dots v_{n+k})$ the corresponding vector in $\mathbb R^{d(n+k)}$. Connecting $v_1,\dots, v_{n+k}$ by straight segments according with the topology $T$ we get an immersion of $T$ into $\mathbb R^d$, calling such an immersion a \textit{network}. Denote by $L(\bar v)$ the sum of the total length of this network. It is widely known that $L$ is a convex function on $\mathbb{R}^{d(n+k)}$.

Now consider the set of local minima of $L$ over the set of $\bar v$ such that $v_i \in \mathcal{A}_i$ for $1 \leq i \leq n$.
Note that every local minimum the corresponding network should satisfy items~(ii) and~(iii) of Observation~\ref{Obs:local} since it is based on local arguments, i.e. does not require that $T$ is properly embedded.

Consider two local minima and let $S_0$ and $S_1$ be the corresponding networks. Let $\{a_i\}$ and $\{b_i\}$ denote the sets of terminals of $S_0$ and $S_1$ respectively enumerated such that $a_i, b_i \in \mathcal{A}_i$; denote the sets of Steiner points of $S_0$ and $S_1$ by $\{c_j\}$ and $\{d_j\}$.

Let us show that $S_0$ and $S_1$ have the same length. Put $e_i (\alpha) = \alpha \cdot a_i + (1-\alpha) \cdot b_i$ and 
$f_j (\alpha) = \alpha \cdot c_j + (1-\alpha) \cdot d_j$ for $\alpha \in [0,1]$.
By the convexity of $\mathcal{A}_i$, a point $e_i(\alpha)$ belongs to $\mathcal{A}_i$ for every $\alpha$.
Hence, for every $\alpha$ the network $S_\alpha$ with the vertices $e_1 (\alpha),\dots,e_n(\alpha), f_1(\alpha),\dots, f_k(\alpha)$ connects $\mathcal{A}_i$ with each other. Now the convexity of $L$ and the fact that $S_0,S_1$ are corresponds to the locally minima of $L$ imply that all $S_\alpha$ have the same length.

Since each two local minima have the same length, each local minimum is a global minimum. In particular, both $S_0$ and $S_1$, and all other $S_\alpha$ are locally minimal trees.

Assume now that the $i$-th terminal of $T$ has degree 1. By the minimality of $S_\alpha$ and Observation~\ref{Obs:local}, the corresponding vertex $e_i(\alpha)$ cannot lie in the interior of $\mathcal{A}_i$; it follows that $a_i = b_i$, since $\mathcal{A}_i$ is strictly convex.

Suppose additionally that the $i$-th terminal is connected to another terminal, say, the $j$-th one. 
By Observation~\ref{Obs:local} the directions of $S_1$ and $S_2$ at $a_i$ coincide (recall that the boundary of $\mathcal{A}_i$ is $C^1$-smooth). Then the nearest point to $a_i = b_i$ in $S_0\cap \mathcal{A}_j$ coincides with the nearest point to $a_i = b_i$ in $S_1 \cap \mathcal{A}_j$; call this point $a_*$. One may now reduce the example by deleting $[a_ia_*[$ from $S_1$ and $S_2$, and replacing $\mathcal{A}_i$ and $\mathcal{A}_j$ with $\mathcal{A} := \{a_*\}$. We get a contradiction with the minimality of $n$.

Now we may assume that $T$ has no adjacent terminals. It follows that $T$ contains a pair of terminals adjacent to the same Steiner point and both having the degree 1. To see this, note that the number of terminals is equal to the number of Steiner points plus 2, hence we can find at least two pairs of terminals such that in each pair the terminals are adjacent to a common Steiner point, because terminals are not connected to each other. Apply this argument to any full component of $T$ which intersects the rest of $T$ by a single vertex. Then at least one of the corresponding two pairs has both terminals of degree 1 in $T$. 

Fix such a pair of terminals, say, the $i$-th and $j$-th one. Let $t$ be the index of the Steiner point they are adjacent to. Note that both $e_i(\alpha)$ and $e_j(\alpha)$ do not depend on $\alpha$ since $a_i,a_j$ are of degree 1. In its turn, the point $f_t(\alpha)$ must belong to the intersection of the straight segment connecting $c_t$ and $d_t$, and the locus of points $p$ such that the angle $a_ipa_j = 2\pi/3$. Such an intersection may contain at most two points, thus all $f_t(\alpha)$ stay the same and $c_t = d_t$. Hence, we can again reduce the example by replacing $\mathcal{A}_i$ and $\mathcal{A}_j$ with $\mathcal{A} = \{c_t\}$ and deleting $a_ic_t$ and $a_jd_t$ from $S_0$ and $S_1$. This again contradicts with the choice of $n$.

\end{proof}

The following example (see Figure~\ref{Fig:3stadiona}) shows that strict convexity is essential.
Let $\mathcal{A}_1$, $\mathcal{A}_2$ and $\mathcal{A}_3$ be stadiums whose sides are parts of an equilateral triangle $\mathcal{B}$. It is well-known (see barycentric coordinate system) that the length of every blue tripod orthogonal to $\mathcal{B}$ is the same and coincides with the length of every dashed polyline, orthogonal to $\mathcal{B}$. It is straightforward to see that all of them are local and global minimizers.

\begin{figure}[h]
        \centering
        \input{pictures/3stadiona.tex}
        \caption{The condition of strict convexity is sufficient in Lemma~\ref{keylemma}}
        \label{Fig:3stadiona}
\end{figure}
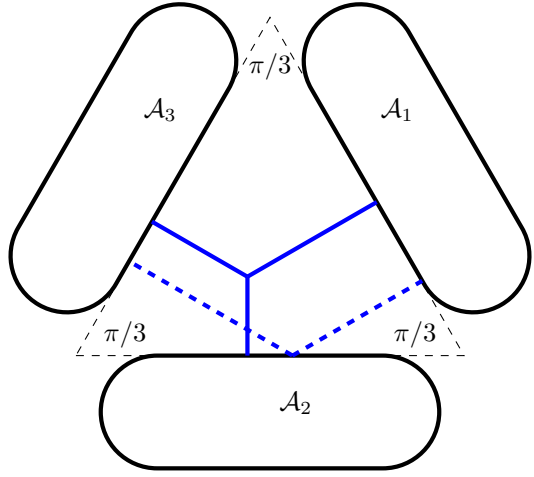

\section{Proof of Theorem~\ref{theocornerpoints}}

\begin{proof}
Let $\xi :=\H(\Sigma)$. Then $\xi$ tends to 0 and $|v_1v_{\infty+1}| = 2r + \xi + o(\xi)$ provided by $N$ tends to infinity. 
Assume that $\Sigma_1 \neq \Sigma$ is a minimizer for $\{v_i\}_{i=1}^{\infty+1}$. Put $J := \mathbb{N} \cup \{\infty,\infty+1\}$.

\paragraph{Step 1: $\Sigma_1$ has no branching points.}
 By the definition $\H(\Sigma_1)\leq  \xi$ and hence
 \[
 \Sigma_1 \subset \overline{B_{r+\xi}(v_1) \cap B_{r+\xi}(v_{\infty +1})}\subset\overline{B_{\sqrt{r\xi}+o(\sqrt{\xi})}(\Sigma)}.
 \]
 The last inclusion holds as 
 \[
 \Sigma_1 \subset B_{o(\xi)}([v_1v_{\infty+1}])\setminus B_r(v_1) \setminus B_r(v_{\infty+1}))\]
 and thus, due to the Pythagorean theorem, 
 \[
 \diam (B_{r+\xi}(v_1) \cap B_{r+\xi}(v_{\infty+1}))=\sqrt{(r+\xi)^2 - \left (r+\frac{\xi}{2} \right)^2}+o(\xi)=\sqrt{r\xi}+o(\sqrt{\xi}).
 \]

But for every point $q\in \overline{B_{\sqrt{r\xi}+o(\xi)}(\Sigma)}$ there exists a cone $Q(q)$ with vertex $q$ and angle $\pi-o_\xi(1)$ such that for every ray $[qp)$ of this cone there holds $\angle v_kqp >\pi/2$ for every $k \in  (\mathbb{N}\setminus \{1\}) \cup \infty $. 

Let $N$ be such that the angle of every $Q(q)$ is greater than $2\pi/3$. 
Let us show $q$ can not be a branching point of $\Sigma_1$. Assume the contrary and let $q$ be a branching point. Then at least one of three segments incident to $q$ belongs to $Q(q)$. 
This segment should end by a branching point: assume the contrary and let $[qx] \subset Q(q)\cap \Sigma_1$ where $x$ is an energetic point. Then by the basic property~(b) there exists a number $l \in J\setminus\{1,\infty+1\}$ (if $l = 1$ or $l = \infty + 1$, then the cone contains two segments of $\Sigma_1$ and we may choose another one) such that $B_r(v_l) \cap \Sigma_1 = \emptyset$ and $|xv_l|=r$ thus $\angle v_lxq \geq \pi/2$. On the other hand by the definition of the cone an angle $\angle v_kqx> \pi/2$. So the sum of angles in the triangle $xv_lq$ is greater than $\pi$; this contradiction shows that $x$ should be a branching point. But all the arguments work also for $x$ as $q$ is an arbitrary point, thus there exists an infinite sequence of branching points in $Q(q)$ which contradicts  Theorem~\ref{theoGT}.

\paragraph{Step 2. The order of circles is the direct.}
Now we know that $\Sigma_1$ is a simple curve. 
Let us show that $\Sigma_1$ contains such a sequence of different points $\{s_i\}_{i = 1}^{\infty }\cup \{s_\infty, s_{\infty+1}\}$ that for every $i \in J$ there holds $|v_is_i|\leq r $ and for every indices $i,j,k \in J$ such that $i<j<k$  the path in $\Sigma_1$ connecting $s_i$ and $s_k$ contains also $s_j$.

An index $i$ is called \textit{important} if $B_r(v_i) \cap \Sigma_1 = \emptyset$. This means that $\partial B_r(v_i) \cap \Sigma_1$ consists of a unique (energetic) point $s_i$. Clearly, $1,\infty+1$ are important and the neighbourhoods of the corresponding $s_1$, $s_{\infty+1}$ in $\Sigma_1$ are segments.
Let $I \subset J$ be the ordered subset of all important indices.
Consider $i,j,k \in I$ and let $i<j<k$.
For a large enough $N$ circles $B_r(v_i)$ and $B_r(v_k)$ intersect or $i = 1, k = \infty+1$. In the first case the path $P_{ik}$ in $\Sigma_1$ from $s_i$ to $s_k$ contains $s_j$: otherwise $P_{ik}$ would separate $s_j$ from the rest of $\Sigma_1$, see Fig.~\ref{Fig:cornerexample1}. In the second case the $\Sigma_1$ is the segment $[s_1s_{\infty+1}]$.
Thus the order of important circles is direct.

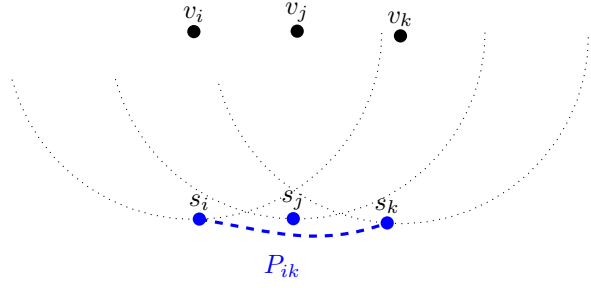
\begin{figure}[h]
        \centering
        \input{pictures/4important.tex}
        \caption{Order of important indices is natural}
        \label{Fig:cornerexample1}
\end{figure}

By basic properties~(b) and~(c) an arbitrary point $x$ of a minimizer $\Sigma_1$ is either the center of a segment contained in $\Sigma_1$ or the center of a tripod contained in $\Sigma_1$ or there is a point $y \in M$ such that $B_r(y) \cap \Sigma_1 = \emptyset$ and $|xy| = r$.

Thus $\Sigma_1$ is a union of segments connecting $s_i$ for consecutive indices from $I$. If $i$ and $j$ are such indices, then
 $\angle v_i s_i s_j\geq \pi/2$, $\angle v_js_j s_i \geq \pi/2$ and the angle between lines $(v_is_i)$, $(v_js_j)$ tends to $0$ when $\xi$ tends to $0$ as $s_i, s_j\in \overline{B_{\sqrt{r\xi}+o(\xi)}(\Sigma)}$.

Now consider index $l$ such that $i < l < j$, so $l\in J \setminus I$ and $[s_is_j] \subset \Sigma_1$. Clearly, as $\Sigma_1 \cap \overline{B_r(v_l)} \neq \emptyset$, $\Sigma_1 \cap B_r(v_i) = \emptyset= \Sigma_1 \cap B_r(v_j)$ then $\Sigma_1 \cap \overline{B_r(v_l)} \subset [s_is_j]$. One can take the projection of $v_l$ onto $[s_is_j]$ as $s_l$.

Therefore there exists a sequence of different points $\{s_i\}_{i \in J} \subset \Sigma_1$ that for every $i$ there holds $|v_is_i|\leq r $ and for every indices $i,j,k \in J$ such that $i < j < k$  the path in $\Sigma_1$ connecting $s_i$, $s_k$ contains also $s_j$.

\paragraph{Step 3. Uniqueness of a local minimum.} 
Let $U := \{u_i\}$, $i \in J$ be a set of points such that $u_i \in \overline{B_r(v_i)}$. 
Consider the length 
\[
L(U) := \sum_{i=1}^\infty |u_i u_{i+1}|.
\]
One can show that by convexity there is a unique minimum of $L$.

Analogously to the proof of Lemma~\ref{keylemma} the set of local minima of $L$ is convex and thus every local minimum is also a global minimum.
Suppose we have two local minima, $\Sigma_1$ and $\Sigma_2$ with $\{s_i\}$ and $\{t_i\}$, $i \in J$. 
In particular, $\H (\Sigma_1) = L(\{s_i\})$ and $\H(\Sigma_2) = L(\{t_i\})$, so $\H (\Sigma_1) = \H (\Sigma_2)$.

Let $u_i(\alpha) = \alpha s_i + (1-\alpha) t_i$, then $\Sigma(\alpha) := \bigcup [u_i(\alpha)u_{i+1}(\alpha)]$ is a global minimizer for every $\alpha \in [0,1]$. Assume that $s_i \neq t_i$ for some $i$ and choose the smallest such $i$.
If $i=1$, then $u_1(1/2)$ lies in $B_r(v_1)$ which is impossible in a minimizer.
So $i > 1$ and by Observation~\ref{Obs:local}(iii)(b) the segments $[s_{i-1}s_i]$ and $[t_{i-1}t_i]$ belong to the same line.
Since $s_i \neq t_i$ index $i$ is important in exactly one of $\Sigma_1$, $\Sigma_2$, without loss of generality in $\Sigma_1$.
Then $\angle s_{i-1}s_is_{i+1} < \pi$ and $\angle t_{i-1}t_it_{i+1} = \pi$.
Now consider $\Sigma (1/2)$ in a neighbourhood of $u_i(1/2)$. Then $u_i(1/2) \in B_r(v_i)$ and $\angle u_{i-1}(1/2)u_i(1/2)u_{i+1}(1/2) < \pi$, which is impossible since we can replace $[u_i(1/2)v] \cup [u_i(1/2)w] $ with $[vw]$ in the minimizer $\Sigma(1/2)$, where $\{v,w\} = \partial B_\varepsilon(u_i(1/2)) \cap \Sigma(1/2)$ and $\varepsilon$ is small enough. Contradiction.

\paragraph{Step 4. $\Sigma$ is the minimum.} 
Recall that $\Sigma$ is polychain with vertices $\{a_i\}_{i \in J}$.
Define $M_i$ as the set consisting of $v_1, \dots, v_i$ and a point $v^i_\infty$ which is the closest point to $v_{\infty+1}$ among the points $v$ such that $\angle a_{i-1}a_iv_i = \angle v_ia_iv$. 
Repeating steps 1--3 for $M_i$ we obtain that there is a unique minimum $\Sigma^i = \cup_{j=1}^{i-1} [a_ja_{j+1}] \cup [a_ia^i_\infty]$, where $a^i_\infty$ is the point in the segment $[a_iv^i_\infty]$ such that $|a^i_\infty v^i_\infty| = r$.

Suppose now that $\Sigma$ is not an $r$-minimizer for $M$, then its length is longer than the length of an $r$-minimizer $\Sigma'$ for $M$ by a positive $\varepsilon$. But for a large enough $i$ we have $|\H(\Sigma_i) - \H(\Sigma)| < \varepsilon/3$ and $|v^i_\infty v_{\infty+1}| < \varepsilon/3$.
It follows from the latter inequality that one can extend $\Sigma'$ by a segment of the length $\eps/3$ such that it will cover $M_i$. But the length of the extended $\Sigma'$ will remain to be strictly smaller length than $\Sigma_i$, which is an absurd.

\end{proof}

\section{Proof of Theorem~\ref{theorem:c11} and Corollaries~\ref{cor:length+curv},~\ref{cor:C11_r_small_enough}}
\label{sec:appendix}

\subsection{Proof of Theorem~\ref{theorem:c11}}

The goal of this section is to prove Theorem~\ref{theorem:c11}. We need the following auxiliary lemmas.

\begin{lemma}
    \label{lemma:crude_area_bound}
    For any curve $\gamma\subset \mathbb R^d$ connecting $x,y$ and any $R>\diam(\gamma)$ the following inequality holds:
    \[
        \HH^d(B_R(\gamma))\geq |x-y| \cdot \omega_{d-1} (R - \diam(\gamma))^{d-1}  + \omega_d R^d.
    \]
\end{lemma}
\begin{proof}
    Without loss of generality we can assume that $y = (0,\dots,0)$ and $x = (t,0,\dots,0)$ for some $t>0$. In this case the lemma follows from the fact that
    \begin{multline*}
        B_R(\gamma)\supset \left \{ (z_1,\dots,z_d)\in \mathbb R^d\ \mid\ 0< z_1<t,\ \left(\sum_{i = 2}^d z_i^2\right)^{1/2}< R - \diam(\gamma)  \right \} \sqcup \\
        \sqcup \{ (z_1,\dots,z_d)\in B_R(y)\ \mid\ z_1\leq 0 \} \sqcup \{ (z_1,\dots,z_d)\in B_R(x)\ \mid\ z_1\geq t \}.
    \end{multline*}
\end{proof}

\begin{lemma}
    \label{lemma:curve_in_cone}
    Let $\gamma\subset \mathbb R^d$ be a rectifiable curve and let $\gamma:[0,\H(\gamma)]\to \mathbb R^d$ be also its arc length parametrization, abusing the notation. Given $t\in (0,\H(\gamma))$ and $R>0$, consider
    \[
        Z_t = \{p\in B_R(\gamma)\ \mid\ \dist(p,\gamma) = |p - \gamma(t)|\}.
    \]
    Assume that we have $\HH^d(B_R(\gamma) ) = \H (\gamma) \omega_{d-1}R^{d-1} + \omega_d R^d$. Then $Z_t$ is contained in some affine hyperplane in $\mathbb{R}^d$.
\end{lemma}
\begin{proof}
    Assume the contrary. Then there exist points $p_1,\dots, p_d\in Z_t$ such that the vectors $p_1-\gamma(t), \dots, p_d - \gamma(t)$ are linearly independent. By the definition $Z_t$ it means that there exist $R>r_1,\dots,r_d>0$ such that
    \begin{itemize}
        \item spheres $\partial B_{r_1}(p_1),\dots, \partial B_{r_d}(p_d)$ intersect transversally at $\gamma(t)$, and
        \item for any $k = 1,\dots,d$, $\gamma\cap B_{r_k}(p_k) = \varnothing$.
    \end{itemize}
    Put $\gamma_1 = \gamma([0,t])$ and $\gamma_2 = \gamma([t,\H(\gamma)])$. We will show that
    \begin{equation}
        \label{eq:cone1}
        \HH^d(B_R(\gamma_1)\cap B_R(\gamma_2)) > \omega_dR^d,
    \end{equation}
    which, by Lemma~\ref{lemma:area_inequality} applied to $\gamma_1$ and $\gamma_2$, implies $\HH^d(B_R(\gamma) ) < \H (\gamma) \omega_{d-1}R^{d-1} + \omega_d R^d$, whence the contradiction.

\begin{figure}[h]
    \centering
    \input{pictures/6lemma42.tex}
    \caption{Choice of $w$ in Lemma~\ref{lemma:curve_in_cone}.}
    \label{pic:6choosingw}
\end{figure}
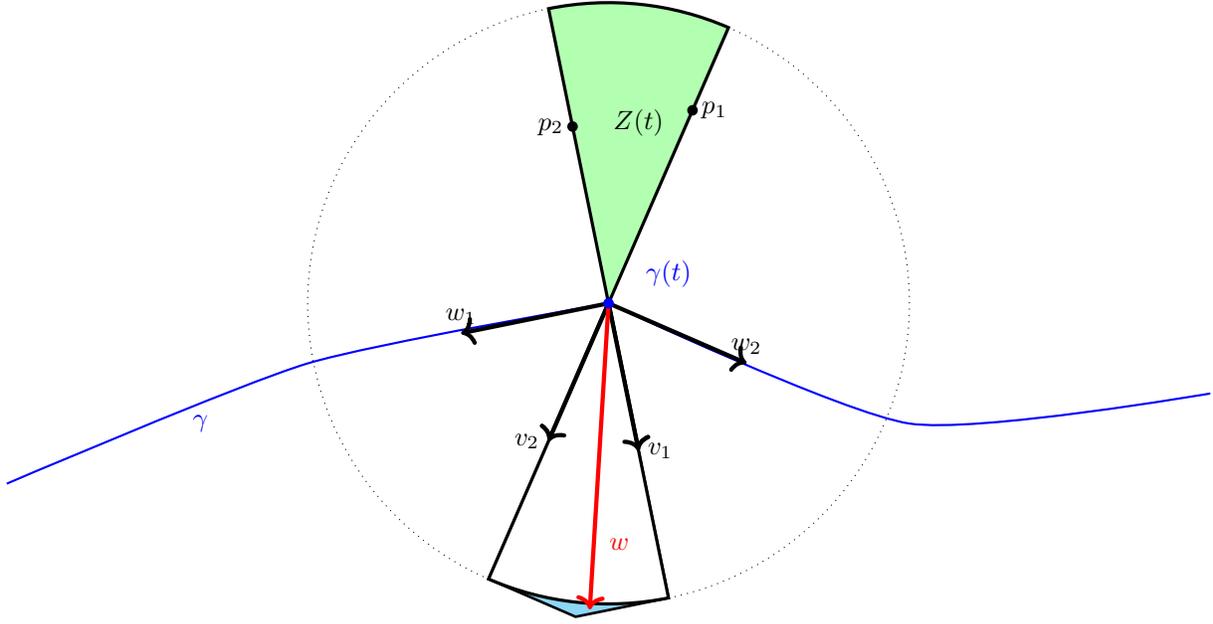

    For $k = 1,\dots, d$, define $v_k = \frac{\gamma(t) - p_k}{|\gamma(t) - p_k|}$ and let $w_1,\dots, w_d$ be the biorthogonal system. Set $w = \frac{w_1+\ldots+w_d}{|w_1 + \ldots+w_d|}$ and $p = \gamma(t) + (R+\eps)w$ for some $\eps>0$. Now, one can check that if $\eps$ is small enough, then for some small $\delta_1,\delta_2>0$ depending on $\eps$ we have $|p -\gamma(t+\delta_1)|<R-\eps$ and $|p - \gamma(t-\delta_2)| < R-\eps$. It follows that
    \[
        B_R(\gamma_1)\cap B_R(\gamma_2) \supset B_R(\gamma(t))\sqcup B_\eps(p)
    \]
    and~\eqref{eq:cone1} follows.
\end{proof}

\begin{lemma}
\label{lemma:good1}
    Let $\gamma\subset \mathbb R^d$ be a simple rectifiable curve parameterized by its arc length, and assume that for any $p\in B_R(\gamma)$ there is a unique closest point on $\gamma$. Then for any $t\in [0,\H(\gamma)]$ we have
    \[
        B_R(\gamma([0,t]))\cap B_R(\gamma([t,\H(\gamma)])) = B_R(\gamma(t)).
    \]
\end{lemma}
\begin{proof}
    Put $\gamma_1 = \gamma([0,t])$ and $\gamma_2 = \gamma([t,\H(\gamma)])$. Assume that the lemma is false, then we can find $p\in B_R(\gamma)$ such that $|p-\gamma(t)|\geq R$ and $\dist(p, \gamma_i) < R$ for $i = 1,2$. Let $t_1\in [0,t)$ and $t_2\in (t,\H(\gamma)]$ be such that $\dist(p, \gamma_i) = |p - \gamma(t_i)|$, and let $l_i = [p\gamma(t_i)]$. Let $l = l_1\cup l_2$, then by the continuity of the length there is a point $q\in l$ such that $\dist(q,\gamma_1) = \dist(q,\gamma_2)$. Assume without loss of generality that $q\in l_1$. Then it is easy to see that $\dist(q,\gamma_1) = |q - \gamma(t_1)|$. Since $t_1\neq t$, it follows that $q$ has two closest points on $\gamma$, so we get a contradiction.
\end{proof}

\begin{lemma}
\label{lemma:goods}
    Let $\gamma\subset \mathbb R^d$ be a simple rectifiable curve parameterized by its arc length, and assume that for any $p\in B_R(\gamma)$ there is a unique closest point on $\gamma$. Then for any $t\in [0,\H(\gamma)]$ the curves $\gamma([0,t])$ and $\gamma([t,\H(\gamma)])$ have the same property.
\end{lemma}
\begin{proof}
    Assume that the statement of lemma is false for some $t\in [0,\H(\gamma)]$ and put $\gamma_1 = \gamma([0,t])$ and $\gamma_2 = \gamma([t,\H(\gamma)])$. By our assumption we have $p\in B_R(\gamma_1)$ (the case of $\gamma_2$ is treated verbatim) there exist two times $t_1<t_2\leq t$ such that $\dist(p,\gamma_1) = |p - \gamma(t_i)|, i = 1,2$. Note that $\dist(p,\gamma)< \dist(p,\gamma(t_1))$, otherwise we would have a contradiction, in particular $\dist(p,\gamma) = \dist(p,\gamma_2)$. Put $l_1 = [p\gamma(t_1)]$, there exists $q\in l_1$ such that $\dist(q,\gamma_1) = \dist(q,\gamma_2)$. Since $\dist(q,\gamma_1) = |p - \gamma(t_1)|$, we find out that there are two closest points on $\gamma$ for $q$, hence we get a contradiction.
\end{proof}

\begin{proof}[Proof of Theorem~\ref{theorem:c11}]
    $(i)\Rightarrow (ii)$. Assume that we can find a point $p\in B_R(\gamma)$ and $t_1<t_2\in [0,\H(\gamma)]$ such that 
    \[
        R' \coloneqq \dist(p,\gamma) = |p - \gamma(t_1)| = |p - \gamma(t_2)|.
    \]
    Note that $0<R'<R$. If $\gamma(t_1) = \gamma(t_2)$, then we take $t_1<t<t_2$ to be any time such that $\gamma(t)\neq \gamma(t_1)$ and set $l$ to be the line passing through $\gamma(t)$ and $\gamma(t_1)$. Otherwise, let $H\subset \mathbb R^d$ be the affine hyperplane passing though $p$ and perpendicular to the line $\gamma(t_1)\gamma(t_2)$, and let the time $t_1<t<t_2$ be such that $\gamma(t)\in H$. Let $L$ be the 2-dimensional plane passing through $\gamma(t_1),\gamma(t_2),\gamma(t)$ and $C = L\cap \partial B_{R'}(p)$. Note that $C$ is a circle of radius at most $R'$, $\gamma(t_1),\gamma(t_2)\in C$ and $\gamma(t)$ does not lie in the interior of $C$; we also have that the triangle $\gamma(t_1)\gamma(t)\gamma(t_2)$ is isosceles (see~Figure~\ref{pic:5planes}). Let $l = H\cap L$ and denote by $O$ the center of $C$.

    Now, for both choices of $l$ one can find a point $q\in l$ and an $\eps>0$ such that 
    \[
        |q-\gamma(t_1)| = |q-\gamma(t_2)| \leq R-\eps,\qquad |q - \gamma(t)|\geq R+\eps.
    \]
    Consider now two curves $\gamma_1 = \gamma([0,t])$ and $\gamma_2 = \gamma([t,\H(\gamma)])$. Note that
    \[
        B_R(\gamma(t))\sqcup B_\eps(q)\subset B_R(\gamma_1)\cap B_R(\gamma_2),
    \]
    hence
    \begin{equation*}
        \HH^d(B_R(\gamma))\leq \HH^d(B_R(\gamma_1)) + \HH^d(B_R(\gamma_2)) - \omega_d(R^d + \eps^d)\leq \H (\gamma) \omega_{d-1}R^{d-1} + \omega_d R^d - \omega_d \eps^d.
    \end{equation*}
    We obtain a contradiction.
    
\begin{figure}[h]
    \centering
    \input{pictures/5planes.tex}
    \caption{Cutting curve $\gamma$ in two pieces whose $R$-neighbourhoods has too large intersection.}
    \label{pic:5planes}
\end{figure}
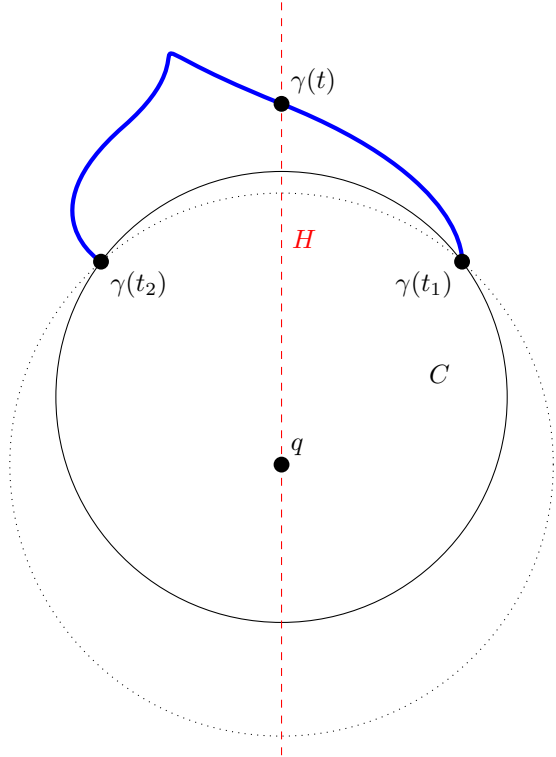

    $(ii)\Rightarrow (i)$. Let us fix some $R>0$ and name a curve $\gamma$ ``good'' if it satisfies $(ii)$ with this $R$. By Lemma~\ref{lemma:goods}, if $\gamma$ is good and $t\in (0,\H(\gamma))$, then $\gamma([0,t])$ and $\gamma([t,\H(\gamma)])$ are good. Fix a big number $n>0$ split $\gamma$ into $n$ subcurves of equal length.
    Applying Lemma~\ref{lemma:good1} and Lemma~\ref{lemma:crude_area_bound} to this splitting we obtain
    \[
        \HH^d(B_R(\gamma))\geq \sum_{k = 0}^{n-1} \Bigl |\gamma\Bigl(k\H(\gamma)/n\Bigr) - \gamma\Bigl((k+1)\H(\gamma)/n\Bigr)\Bigr|\cdot \omega_{d-1}\cdot (R - \H(\gamma)/n)^{d-1} + \omega_dR^d.
    \]
    Passing $n\to \infty$ we obtain the inequality
    \[
        \HH^d(B_R(\gamma))\geq \H(\gamma) \omega_{d-1}R^{d-1} + \omega_dR^d.
    \]
    The reversed inequality is given by Lemma~\ref{lemma:area_inequality}.

    $(i)\& (ii) \Rightarrow (iii).$ We need to prove that $\gamma$ has the curvature radius at least $R$. Thanks to $(ii)$ we can define the function $F:B_R(\gamma)\to \gamma$ by
    \[
        |p - F(p)| = \dist(p,\gamma).
    \]
    Note that $F$ is continuous; indeed, if $p_k\to p$, but $F(p_k)\not\to F(p)$, then $p$ will have two closest points on $\gamma$. Fix $t\in (0,\H(\gamma))$ and consider the set $Z_t = F^{-1}(\gamma(t))$. We claim that $B_R(\gamma(t))\smallsetminus Z_t$ is not connected. Indeed, $F(B_R(\gamma(t))\smallsetminus Z_t)$ is disconnected, since it contains $\gamma(t-\eps)$ and $\gamma(t+\eps)$ for some small $\eps>0$, we conclude with the fact that $F$ is continuous.

    From the fact that $B_R(\gamma(t))$ is connected but $B_R(\gamma(t))\smallsetminus Z_t$ is not we can deduce easily that $Z_t$ has Hausdorff dimension at least $d-1$. It follows that there exist points $p_1,p_2,\dots,p_{d-1}\in Z_t$ such that the vectors $p_1-\gamma(t),\dots,p_{d-1} - \gamma(t)$ are linearly independent. By Lemma~\ref{lemma:curve_in_cone} and $(i)$ this implies that for any $p_d\in Z_t$ we have
    \[
        p_d - \gamma(t)\in V \coloneqq \mathrm{span}\{ p_1 - \gamma(t),\dots, p_{d-1} - \gamma(t) \}.
    \]
    Using again $B_R(\gamma(t))\smallsetminus Z_t$ is not connected we conclude that $Z_t = V\cap B_R(\gamma(t))$.

    Introduce the notation
    \[
        \mathcal{C} = F^{-1}\Bigl(\gamma\Bigl((0,\H(\gamma))\Bigr)\Bigr).
    \]
    We construct a vector field $w$ on $\mathcal{C}$ as follows. Given $p\in \mathcal{C}$, let $\bar{w}(p)\in \mathbb RP^{d-1}$ be the direction orthogonal to $Z_t$, where $t$ is such that $\gamma(t) = F(p)$. We claim that $\bar{w}$ is continuous. Indeed, by the definition $Z_{t_1}\cap Z_{t_2} = \varnothing$; but if for a sequence $p_k\to p$ we do not have $\bar{w}(p_k)\to \bar{w}(p)$, then $Z_{t_k}\cap Z_t\neq \varnothing$ from a certain moment, where $\gamma(t_k) = F(p_k)$.
    
    Recall that $\partial B_1(0)\to \mathbb RP^{d-1} = \mathbb{R}^d/R$ is a covering map. Thus, by lifting property we can extend the map $\bar{w}$ to a map 
    \[
        w: \mathcal{C}\to \partial B_1(0),   
    \]
    such that $w$ is continuous. Let $\Phi_p(t)$ be the solution to the equation
    \[
        \Phi_p(0) = p,\qquad \frac{d}{dt} \Phi_p(t) = w(\Phi_p(t)).
    \]
    We claim that after replacing $w$ with $-w$ if necessary we have $\Phi_{\gamma(t)}(s) = \gamma(t+s)$ for any $t,s>0$ such that $t+s< \H(\gamma)$. Indeed, consider the function $h(x) = \dist(x,\gamma)$. One can show that for any $x\in \mathcal{C}\smallsetminus \gamma$ the function $h$ is differentiable and $\nabla h(x) = \frac{x - F(x)}{|x - F(x)|}$. It follows that $h(\Phi_p(t))$ is constant in $t$ for any $p\in \mathcal{C}\smallsetminus\gamma$, and therefore for any $p\in \mathcal{C}$ by continuity. We conclude that $\Phi_{\gamma(t)}(s)\in \gamma$ for small $s$. This and the fact that $|w| = 1$ imply the claim. 
    
    It follows that $\gamma\in C^1$ and $\gamma'(t) = w(\gamma(t))$ for any $t\in (0,\H(\gamma))$. The fact that $w(\gamma(t))$ is $R^{-1}$-Lipschitz easily follows from the fact that discs $Z_t$ are mutually disjoint, which finalizes the proof.

    Note that $(iii)\Rightarrow (ii)$, hence we proved the equivalence between $(i), (ii), (iii)$.
\end{proof}

\subsection{{Proof of Corollaries~\ref{cor:length+curv},~\ref{cor:C11_r_small_enough}}}

We begin by generalizing~\cite[Proposition~2.4]{tilli2010some} to $d$-dimensional space:

\begin{lemma}
    \label{lemma:tillis_lemmas}
    Let $\gamma$ be a rectifiable curve in $\mathbb R^d$ of length $l\leq \pi R$ and curvature radius at least $R$ for some $R>0$, let $\gamma:[0,l]\to \mathbb R^d$ be its arc length parametrization. Assume that $\gamma(0) = (R,0,\dots,0)$ and $\gamma'(0) = (0,\dots, 0,1)$. Then $\gamma\cap B_R(0) = \varnothing$, where $B_R(0)$ denotes the open ball.
\end{lemma}
\begin{proof}
    Our arguments essentially repeat the arguments from the proof of~\cite[Lemma~2.2,~Lemma~2.3]{tilli2010some}, modulo some minor adjustments needed to work out the non-planar setup. We include them for the sake of completeness.
    
    Without loss of generality we can assume that $R = 1$ and $d\geq 2$. In this case it is enough to prove that
    \begin{equation}
        \label{eq:tillis_lemmas0}
        |\gamma(l)|^2 = \left|\gamma(0) + \int_0^l\gamma'(t)\,dt\right|^2 \geq 1.
    \end{equation}
    We begin by proving that
    \begin{equation}
        \label{eq:tillis_lemmas1}
        \left| \int_0^l \gamma'(t)\,dt \right|^2\geq 2(1-\cos l).
    \end{equation}
    To this end write
    \begin{equation*}
        \left| \int_0^l \gamma'(t)\,dt \right|^2 = \int_0^l\int_0^l \gamma'(t)\cdot \gamma'(s)\,dtds \geq \int_0^l\int_0^l \cos(t-s)\,dtds = 2(1-\cos l),
    \end{equation*}
    where the inequality follows from the fact that $\gamma'(t)$ is 1-Lipshitz as a map from $[0,l]$ to the unit sphere endowed with the standard metric. Write
    \[
        \gamma'(t) = (x_1(t), \dots, x_d(t)).
    \]
    By the same reason as above we have $x_d(t)\geq \cos t$ for all $t\in [0,l]$, whence
    \begin{equation}
        \label{eq:tillis_lemmas2}
        \left(\int_0^lx_d(t)\,dt\right)^2\geq (\sin l)^2.
    \end{equation}
    We now prove~\eqref{eq:tillis_lemmas0} by considering three separate cases.

    \emph{Case 1: $l\leq \pi/2$}. In this case we can use the Lipshitzness of $\gamma'$ again to conclude that $x_1(t)\leq \sin(t)$ and thus
    \begin{equation}
        \label{eq:tillis_lemmas3}
        \left(1+\int_0^lx_1(t)\,dt\right)^2\geq (\cos l)^2
    \end{equation}
    which implies~\eqref{eq:tillis_lemmas0} if combined with~\eqref{eq:tillis_lemmas2}.

    \emph{Case 2: $\left|\int_0^l x_1(t)\,dt\right|\leq 1-\cos l$}. In this case we can deduce~\eqref{eq:tillis_lemmas0} expanding
    \[
        \left|\gamma(0) + \int_0^l\gamma'(t)\,dt\right|^2 = \left| \int_0^l \gamma'(t)\,dt \right|^2 + 2\int_0^lx_1(t)\,dt + 1
    \]
    and using~\eqref{eq:tillis_lemmas1}.

    \emph{Case 3: $l>\pi/2$ and $\left|\int_0^l x_1(t)\,dt\right|> 1 - \cos l$}. Since $\cos l < 0$ this implies~\eqref{eq:tillis_lemmas3} and we can proceed as in the first case.
\end{proof}

\begin{proof}[Proof of Corollary~\ref{cor:length+curv}]
    Let us prove that $\gamma$ satisfies the condition from item (ii) of Theorem~\ref{theorem:c11}. We prove it by contradiction. Let, as usual, $l$ be the length of $\gamma$ and $\gamma:[0,l]\to \mathbb R^d$ denote the arc length parametrization. Assume that we can find a point $p$ in the $R$-neighborhood of $\gamma$ which has two closest points on $\gamma$. Without loss of generality we can assume that these two points are $\gamma(0)$ and $\gamma(l)$. Recall that $\gamma'(0)$ and $\gamma'(l)$ denote the one-sided tangent vectors to $\gamma$ oriented along $\gamma$. Note that
    \[
        (p-\gamma(0))\cdot \gamma'(0) \leq 0,\qquad (p - \gamma(l))\cdot \gamma'(l)\geq 0.
    \]
    Let $R' = |\gamma(0) - p|$.
    
    Assume that $(p-\gamma(0))\cdot \gamma'(0) = (p - \gamma(l))\cdot \gamma'(l) = 0$, i.e. the sphere $\partial B_{R'}(p)$ is tangent to $\gamma$ at $\gamma(0)$ and $\gamma(l)$. It follows that the sphere
    \[
        \partial B_R\left(\gamma(0) + \frac{R}{R'}(p - \gamma(0))\right)
    \]
    is tangent to $\gamma$ at $\gamma(0)$ and contains $\gamma(l)$ in its interior which contradicts with Lemma~\ref{lemma:tillis_lemmas}. Note in particular that these arguments imply that $\gamma$ is simple.

    Assume now that $(p-\gamma(0))\cdot \gamma'(0) < 0$. Applying a translation and rotation if necessary we can assume that $\gamma(0) = 0$ and $\gamma'(0) = (0,\dots, 0,1)$. Put
    \[
    \begin{split}
        &S = \{0\}\cup \bigcup_{x = (x_1,\dots, x_{d-1},0)\ \colon |x| = R} B_R(x),\\
        &C = \{x \in B_R(0)\ \mid\ x_d<0 \},\\
        &\mathbb H^+ = \{x\in \mathbb R^d\ \mid\  x_d > 0\},\\
        &\mathbb H^- = \{x\in \mathbb R^d\ \mid\  x_d < 0\}.
    \end{split}
    \]
    By our assumption $p\in C$, and by Lemma~\ref{lemma:tillis_lemmas}
    \begin{equation}
        \label{eq:lc1}
        \gamma((0,l])\cap S = \varnothing.
    \end{equation}
    Note that 
    \[
        \partial B_{R'}(p)\smallsetminus S\subset \mathbb H^-.
    \]
    With the definition of $p$ this implies that $\gamma(l)\in \mathbb H^-$, therefore, by~\eqref{eq:lc1} and the fact that $\gamma'(0) = (0,\dots, 0,1)$, the curve $\gamma$ must cross the hyperplane $\{x_d = 0\}$ at some point outside $S$. Let $l_0\in (0,l)$ be the first time it happens, note that $\gamma((0,l_0))\subset \mathbb H^+$. Define $\Phi: \mathbb R^d\to \mathbb R^2$ by $\Phi(x_1,\dots, x_d) = \left(\sqrt{x_1^2+\ldots+x_{d-1}^2}, x_d\right)$ and put
    \[
        \gamma_1 = \Phi(\gamma([0,l_0])).
    \]
    Then $\gamma_1$ is a planar curve connecting the origin with a point on the positive half of the X-axis within the upper half-plane. Notice that $\gamma_1$ does not intersect the open disc $\{(x-R)^2 + y^2 < R^2\}$ because of~\eqref{eq:lc1}, and notice also that the length of $\gamma_1$ is at most $l_0$ because $\Phi$ is 1-Lipshitz.

    Let $\gamma_2$ be the mirror image of $\gamma_1$ under the reflection with respect to the X-axis, and let $\bar{\gamma} = \gamma_1\cup \gamma_2$. Then $\bar{\gamma}$ is a closed curve and one of the bounded connected components of $\mathbb R^2\smallsetminus \bar{\gamma}$ contains the disc $\{(x-R)^2 + y^2 < R^2\}$. It follows by the isoperimetric inequality that the length of $\bar{\gamma}$ is at least $2\pi R$. But $l_0<l\leq \pi R$ by our assumptions, which leads to a contradiction.
\end{proof}

\begin{proof}[Proof of Corollary~\ref{cor:C11_r_small_enough}]
    Let $R>0$ be the minimal curvature radius of $\gamma$, let $l$ be the length of $\gamma$ and let $\gamma:[0,l]\to \mathbb R^d$ be the arc-length parametrization. Put
    \[
        \eps_1 = \frac{1}{2}\min\{|\gamma(t) - \gamma(s)|\ \mid\ 0\leq t\leq s+\pi R \leq l \};
    \]
    if the set on the right-hand side is empty, then put $\eps = R$, in the other case put $\eps = \min(\eps_1, R)$. Note that $\eps>0$ because $\gamma$ is simple. Corollary~\ref{cor:length+curv} and the definition of $\eps_1$ imply that $\gamma$ satisfies the condition from item~(ii) of Theorem~\ref{theorem:c11} with $R = \eps$. The corollary follows.
\end{proof}

\paragraph{Acknowledgements.} The authors are grateful to Fedor Petrov and Alexandr Polyanskii for useful discussions. 
The research is supported by <<Native towns>>, a social investment program of PJSC <<Gazprom Neft>>. The authors condemn the Russian invasion of Ukraine.
\bibliographystyle{plain}
\bibliography{main}

\end{document}

%% file: pictures/1nonunique.tex
\begin{tikzpicture}
    \def\r{1.5cm}
    \draw[ultra thick, blue]
        (-\r, \r) coordinate(x1) node[black, above right]{$1$} --++ (-60:{\r/cos(30)}) coordinate (x5);
    \draw[ultra thick, blue]
        (\r,\r) coordinate(x2) node[black, above left]{$2$} --++ (-120:{\r/cos(30)}) coordinate (x6);
    \draw[ultra thick, blue]
        (\r, -\r) coordinate(x3) node[black, below left]{$3$} --++ (120:{\r/cos(30)});
    \draw[ultra thick, blue]
        (-\r,-\r) coordinate(x4) node[black, below right]{$4$} --++ (60:{\r/cos(30)});
    \draw[ultra thick, blue]
        (x5) -- (x6);
    \foreach \x in{1,2,...,6}{
        \fill (x\x) circle (2pt);
    }
\end{tikzpicture}
\hspace{2cm}
\begin{tikzpicture}
    \def\r{1.5cm}
    \draw[ultra thick, blue]
        (-\r, \r) coordinate(x1) node[black, above right]{$1$} --++ (-30:{\r/cos(30)}) coordinate (x5);
    \draw[ultra thick, blue]
        (\r,\r) coordinate(x2) node[black, above left]{$2$} --++ (-150:{\r/cos(30)});
    \draw[ultra thick, blue]
        (\r, -\r) coordinate(x3) node[black, below left]{$3$} --++ (150:{\r/cos(30)}) coordinate (x6);
    \draw[ultra thick, blue]
        (-\r,-\r) coordinate(x4) node[black, below right]{$4$} --++ (30:{\r/cos(30)});
    \draw[ultra thick, blue]
        (x5) -- (x6);
    \foreach \x in{1,2,...,6}{
        \fill (x\x) circle (2pt);
    }
\end{tikzpicture}

%% file: pictures/2corner.tex
    \hfill  
   \begin{tikzpicture}[scale=15]
    
    \coordinate(O) at (0,0);
    \coordinate(x1) at ({cos(2r)},{sin(2r)});
    \coordinate(x2) at ({cos(1.6r)},{sin(1.6r)});
    \coordinate(x3) at ({cos(1.4r)},{sin(1.4r)});
    \coordinate(x4) at ({cos(1.3r)},{sin(1.3r)});
    \coordinate(x5) at ({cos(1.25r)},{sin(1.25r)});
    \coordinate(x6) at ({cos(1.225r)},{sin(1.225r)});
    \coordinate(xN) at ({cos(1.15r)},{sin(1.15r)});

    \foreach \x in{1,2,...,6}{
    \coordinate(y\x) at ($(x\x)!-0.1!(O)$);
    }    
    \coordinate(yy) at ($(xN)!-0.1!(O)$);
    \coordinate(yN) at ($(xN)!-2.2!(x6)$);
    \coordinate(xNN) at ($(xN)!-0.8!(x6)$);
    \coordinate(y0) at ($(x1)!-0.28!(x2)$);

    \draw [shift={(O)}]  plot[domain=1:2.141,variable=\t]({cos(\t r)},{sin(\t r)});
    
    \foreach \x in{2,3,...,6}{
    \draw[dotted] (x\x) -- (y\x);
    }
    \draw[dotted] (x1) -- (y0);
    \draw[dotted] (xN) -- (yN);
    \draw[dotted] (xN) -- (yy);

    \foreach \x in{2,3,...,5}{
            \draw (y\x) node[above]{$v_\x$};
            \edef\y{\x}
            \pgfmathparse{\y+1}
            \edef\y{\pgfmathresult}
        }

    \foreach \x in{2,3,...,5}{
            \draw [dotted, shift={(y\x)}]  plot[domain={3.4-\x*0.1}:{6.3-\x*0.05},variable=\t]({cos(\t r)*0.1},{sin(\t r)*0.1});
        }

    \foreach \x in{1,2,...,5}{
            \draw (x\x) node[above]{$a_\x$};
            \edef\y{\x}
            \pgfmathparse{\y+1}
            \edef\y{\pgfmathresult}
            \draw [ultra thick, blue] (x\x) -- (x\y);
        }

    \draw [ultra thick, blue] (xN) -- (xNN);

    \draw (x6) node[above right]{$a_6$};
    \draw (xN) node[below]{$a_{\infty}$};
    \draw (y6) node[above right]{$v_6$};
    \draw (y0) node[above right]{$v_1$};
    \draw (yy) node[above right]{$v_\infty$};
    \draw (yN) node[above right]{$v_{\infty+1}$};
    \draw (xNN) node[above right]{$a_{\infty+1}$};
    
    \draw [line width=2pt, dash dot, blue] (x6) -- (xN);
    \draw [line width=2pt, dash dot] (y6) -- (yy);

    \fill [blue] (xN) circle (0.2pt);
    \fill[blue] (xNN) circle (0.2pt);
    
    \foreach \x in{1,2,...,6}{
        \fill [blue] (x\x) circle (0.2pt);
    }
    \foreach \x in{2,...,6}{
        \fill (y\x) circle (0.2pt);
    }
    \fill (y0) circle (0.2pt);
    \fill (yN) circle (0.2pt);
    \fill (yy) circle (0.2pt);
    
\end{tikzpicture}
    \hfill \ 

%% file: pictures/3stadiona.tex
    \begin{tikzpicture}[scale=1.5]
    
    \draw[ultra thick] (1,-1)  arc(90:-90:.5) -- (-1, -2) arc(270:90:.5)  -- cycle;
    \draw[ultra thick, rotate=120] (1,-1)  arc(90:-90:.5) -- (-1, -2) arc(270:90:.5)  -- cycle;
    \draw[ultra thick, rotate=240] (1,-1)  arc(90:-90:.5) -- (-1, -2) arc(270:90:.5)  -- cycle;
    \draw[dashed] 
        (-1.72, -1) node[above=8pt, right=7pt]{$\pi/3$} -- 
        (1.72, -1) node[above=8pt, left=7pt]{$\pi/3$} --
        (0, 2) node[below=10pt]{$\pi/3$} -- cycle;

    \draw[ultra thick,blue] (-.2, -.3) -- (-.2, -1);
    \draw[ultra thick,blue,rotate=120] (-.16, .32) -- (-.16, -1);
    \draw[ultra thick,blue,rotate=240] (.36, -.02) -- (.36, -1);
    \draw[ultra thick,blue,dashed,rotate=120] (-.966, 0.33) -- (-.966, -1);
    \draw[ultra thick,blue,dashed,rotate=240] (.766, 0.67) -- (.766, -1);

    \draw (0.9,1) node[above right]{$\mathcal{A}_1$};
    \draw (-1.2,1) node[above right]{$\mathcal{A}_3$};
    \draw (0,-1.6) node[above right]{$\mathcal{A}_2$};

    \end{tikzpicture}
    

%% file: pictures/4important.tex
    \hfill  
   \begin{tikzpicture}[scale=25]
    
    \coordinate(O) at (0,0);
    \coordinate(xi) at ({cos(1.6r)},{sin(1.6r)});
    \coordinate(xj) at ({cos(1.55r)},{sin(1.55r)});
    \coordinate(xk) at ({cos(1.5r)},{sin(1.5r)});

    \foreach \x in{i,j,k}{
    \coordinate(y\x) at ($(x\x)!-0.1!(O)$);
    \draw [dotted, shift={(y\x)}]  plot[domain={3.4}:{6.3},variable=\t]({cos(\t r)*0.1},{sin(\t r)*0.1});
    \draw (y\x) node[above]{$v_\x$};
    \fill (y\x) circle (0.1pt);
    \fill [blue] (x\x) circle (0.1pt);
    \draw (x\x) node[above]{$s_\x$};
    }    

    \draw[very thick, blue, dashed] (xi) to [out=-10, in=-160] (xk);
    \draw[blue] (0.015,0.985) node[below] {$P_{ik}$};

\end{tikzpicture}
    \hfill \ 

%% file: pictures/6lemma42.tex
\begin{tikzpicture}[scale=4]

        \coordinate (A) at (-2,0);
        \coordinate (AB) at (-1,0.4);
        \coordinate (B) at (0,0.6);
        \coordinate (BC) at (1,0.2);
        \coordinate (C) at (2,0.3);

\draw [blue, thick] plot [smooth, tension=0.4] coordinates { (A) (AB) (B)};
\draw [blue, thick] plot [smooth, tension=0.4] coordinates { (B) (BC) (C)};

\draw (B) --++ (281.5:1) coordinate (u1);
\draw (B) --++ (246.5:1) coordinate (u2);
\draw [ultra thick, ->] (B) --++ (281.5:.5) coordinate (v1);
\draw [ultra thick, ->] (B) --++ (246.5:.5) coordinate (v2);
\draw (u1) --++ (191.5:0.315) coordinate (x) -- (u2);

\draw[very thick,fill=green!30,yshift=0.6cm] (0,0) -- (66.5:1) arc(66.5:101.5:1) -- cycle;

\filldraw[line width=1pt, fill=cyan!40] (B) -- (u1) -- (x) -- (u2) -- (B);

\draw [ultra thick, ->] (B) --++ (191.5:.5) coordinate (w1);
\draw [ultra thick, ->] (B) --++ (336.5:.5) coordinate (w2);

\draw (B) --++ (66.5:0.7) coordinate (p1);
\draw (B) --++ (101.5:0.6) coordinate (p2);

\draw [dotted] (B) circle (1);
\node at (0.1, 1.2) {$Z(t)$};
\node [blue] at (0.2, 0.7) {$\gamma(t)$};
\node [above] at (w2) {$w_2$};
\node [above] at (w1) {$w_1$};
\node [left] at (v2) {$v_2$};
\node [right] at (v1) {$v_1$};
\node [right] at (p1) {$p_1$};
\node [left] at (p2) {$p_2$};
\node [left,blue] at (-1.3,.2) {$\gamma$};

\draw[very thick,fill=white,yshift=0.6cm] (0,0) -- (246.5:1) arc(246.5:281.5:1) -- cycle;

\node [left,red] at (0.1,-0.2) {$w$};

\draw [ultra thick, ->] (B) --++ (281.5:.5) coordinate (v1);
\draw [ultra thick, ->] (B) --++ (246.5:.5) coordinate (v2);

\draw [ultra thick, ->, red] (B) --++ (266.5:1.02) coordinate (w);

\fill (p1) circle (0.5pt);
\fill (p2) circle (0.5pt);
\fill [blue] (B) circle (0.5pt);



\end{tikzpicture}

%% file: pictures/5planes.tex
\begin{tikzpicture}[scale=3]
       \draw (0,0) circle (1);

        \coordinate (x) at (0.8,0.6);
        \coordinate (y) at (-0.8,0.6);
        \coordinate (z) at (0,1.3);
        \coordinate (t) at (-0.7,1.2);

        \coordinate (q) at (0,-0.3);

       \draw [blue, ultra thick] plot [smooth, tension=2] coordinates { (x) (z) (t) (y)};

 \draw[dotted] (q) circle (1.20416);

        \fill (x) circle (1pt) node[below left] {$\gamma(t_1)$};
        \fill (y) circle (1pt) node[below right] {$\gamma(t_2)$};
        
        \draw [red, dashed] (0,1.75) -- (0,-1.6);

        \fill (z) circle (1pt) node[above right] {$\gamma(t)$};
        \fill (q) circle (1pt) node[above right] {$q$};

        \node  at (0.7, 0.1) {$C$};
        \node [red] at (0.1, 0.7) {$H$};

\end{tikzpicture}